\documentclass[7pt]{amsart}

\usepackage[utf8]{inputenc}
\usepackage{amsthm}
\usepackage{amsmath}
\usepackage{pgfplots}
\usetikzlibrary{trees}
\usepackage{mathrsfs}
\usepackage{csquotes}
\usepackage{url}
\usepackage{tikz-cd}
\usepackage[cal=boondox]{mathalfa}
\usepackage{hyperref}
\usepackage{mathtools}
\usepackage{xcolor}
\usepackage{tikz}
\usepackage{float}
\usepackage{eucal}
\usepackage{lipsum}
\usepackage{cancel}
\usepackage{amsfonts}
\usepackage{amssymb}
\usepackage{enumitem}
\theoremstyle{plain}
\newtheorem{theorem}{Theorem}[section]
\theoremstyle{definition}
\newtheorem{definition}[theorem]{Definition}

\newtheorem{corollary}[theorem]{Corollary}
\newtheorem{lemma}[theorem]{Lemma}
\newtheorem{proposition}[theorem]{Proposition}
\newtheorem{conjecture}[theorem]{Conjecture}

\newtheorem{remark}[theorem]{Remark}

\usepackage{multicol}
\usepackage{listings}
\usepackage{xcolor}
\definecolor{codegreen}{rgb}{0,0.6,0}
\definecolor{codegray}{rgb}{0.5,0.5,0.5}
\definecolor{codepurple}{rgb}{0.58,0,0.82}
\definecolor{backcolour}{rgb}{0.95,0.95,0.92}

\newcommand{\hsep}{\hspace{3pt}}

\newcommand{\vsep}{

	\vspace{7pt}

}

\newcommand{\tSyl}{\mathrm{Syl}}

\newcommand{\tker}{\mathrm{ker}}

\newcommand{\defeq}{\coloneqq}

\newcommand{\bC}{\mathbb{C}}

\newcommand{\bN}{\mathbb{N}}

\newcommand{\bP}{\mathbb{P}}
\newcommand{\bQ}{\mathbb{Q}}
\newcommand{\bR}{\mathbb{R}}

\newcommand{\bZ}{\mathbb{Z}}

\newcommand{\cH}{\mathcal{H}}

\newcommand{\cP}{\mathcal{P}}

\usepackage{longtable}

\newcommand{\modsymb}[2]{\lambda({#1},{#2})}
\newcommand{\GU}[1]{\left(\bZ/{#1}\bZ\right)^*}

\newcommand{\tTor}{\mathrm{Tor}}

\newcommand{\tord}{\mathrm{ord}}
\newcommand{\trk}{\mathrm{rk}}

\DeclareFontFamily{U}{wncy}{}
\DeclareFontShape{U}{wncy}{m}{n}{<->wncyr10}{}
\DeclareSymbolFont{mcy}{U}{wncy}{m}{n}
\DeclareMathSymbol{\Sha}{\mathord}{mcy}{"58} 
\DeclareMathSymbol{\Shamalo}{\mathord}{mcy}{"58}

\lstdefinestyle{mystyle}{
	commentstyle=\color{codegreen},
	numberstyle=\tiny\color{codegray},
	basicstyle=\ttfamily\footnotesize,
	breakatwhitespace=false,         
	breaklines=true,                 
	captionpos=b,                    
	keepspaces=true,                 
	numbers=left,                    
	numbersep=5pt,                  
	showspaces=false,                
	showstringspaces=false,
	showtabs=false,                  
	tabsize=2
}
\title{Numerical study of refined conjectures of the BSD type}
\date{}
\author{Juan-Pablo Llerena-C\'ordova}
\email{juan.llerena@ug.uchile.cl}

\begin{document}

\begin{abstract}
In 1987, Mazur and Tate stated conjectures which, in some cases, resemble the classical Birch-Swinnerton-Dyer conjecture and its $p$-adic analog. We study experimentally three conjectures stated by Mazur and Tate using SageMath. Our findings indicate discrepancies in some of the original statements of some of the conjectures presented by Mazur and Tate. However, a slight modification on the statement of these conjectures does appear to hold.
\end{abstract}

\maketitle

\section*{Introduction}
Let $E/\bQ$ be an elliptic curve over the field of rational numbers of conductor $N$. By the Modularity Theorem (\cite{Wil95}, \cite{TW95}, \cite{BCDT01}), we know that $E$ is a modular elliptic curve. Denote by $f_E\in S_2(\Gamma_0(N))$ the rational newform of weight $2$ attached to $E$ by the Modularity Theorem. We define the ``$+$'' modular symbol attached to $E$ at $\frac{a}{b}\in \bQ$ as
\begin{equation*}
	\lambda^+(a,b) \defeq \frac{\pi i}{\Omega_E}\left(\int_{i\infty}^{\frac{a}{b}} f_E(z) dz + \int_{i\infty}^{\frac{-a}{b}} f_E(z) dz\right),
\end{equation*}
where $\Omega_E \defeq \frac{1}{2} \int_{E(\bR)}|\omega_E|$ and $\omega_E$ denotes a N\'eron differential of $E$. As a consequence of the Theorem of Manin-Drinfel'd (\cite{Man72}, \cite{Dri73}) we know that $\lambda^+(a,b)\in \bQ$ for all $\frac{a}{b}\in \bQ$.
\vsep 
Now, by the work of Tate \cite{Tat74} we know that if $E$ has split multiplicative reduction at a prime $p$, there exists a unique $q_p\in \bQ_p^*$, called the \emph{$p$-adic period of $E$}, such that
\[
	E(\bC_p) \cong \bC_p^*/q_p^{\bZ},
\]
as rigid analytic spaces and $\tord_p(q_p) > 0$, where $\tord_p$ denotes the $p$-adic valuation normalized at $p$ \textit{i.e.} $\tord_p(p) = 1$.
\vsep
In \cite{MT87}, Mazur and Tate defined, to which we will refer as, the \emph{Mazur-Tate element at layer $M$} in terms of the modular symbols attached to $E$ (see Definition \ref{Mazur-Tate}).
\vsep
Mazur and Tate stated refined conjectures of the ``Birch-Swinnerton-Dyer type'' concerning the Mazur-Tate element which, under certain restrictions, resemble the classical Birch-Swinnerton-Dyer conjectures \cite{BS65} and its $p$-adic analog, introduced in \cite{MTT96}. Furthermore, in certain cases, the refined conjectures of the Birch-Swinnerton-Dyer type provide information on multiple $p$-adic periods simultaneously, along with the ``interaction'' between them (see \cite[p. 749]{MT87}). This is in contrast to the $p$-adic Birch-Swinnerton-Dyer conjectures which, in some cases, provide information of only one $p$-adic period \cite[Chap. II, \S10]{MTT96}.
\vsep
Besides the numerical evidence presented by Mazur and Tate at the moment of stating their conjectures, most of the numerical calculations found in the literature has been done by Portillo-Bobadilla, in his Ph.D. thesis \cite{Pob04} and in a subsequent article \cite{Pob19}. Specifically, Portillo-Bobadilla studied a slight variation of Conjecture 4 of \cite{MT87}. The objective of the present article is to expand the numerical evidence of some of the conjectures presented in \cite{MT87}.
\vsep
We are mainly interested in Conjecture $6$ of \cite{MT87}, particularly when the layer $M$ is a prime number $p$ for which $E$ has split multiplicative reduction. For an easier implementation in SageMath, we will show in Section $2$ that Conjecture $6$ of \cite{MT87} has a ``multiplicative'' formulation, in the aforementioned case (see Lemma \ref{equivalence_conjectures}). This ``multiplicative'' formulation can be stated as follows.
\begin{conjecture}\label{conjintro}
	Assume that $E$ has split multiplicative reduction at a prime $p$. Define $G_p \defeq \left(\bZ/p\bZ\right)^*/\left<-1\right>$. Let $R\subseteq \bQ$ be the smallest subring such that $\{\lambda^+(a,p)\}_{a\in G_p}\subseteq R$ and $\frac{\lambda^+(0,1)}{2\tord_p(q_p)}\in R$. Also let $\tilde{q}_p \defeq q_p/p^{\tord_p(q_p)} \in \bZ_p^*$, so $\tilde{q}_p$ can be naturally viewed in $\GU{p}$. Define the set of prime numbers $S\subseteq \bN$ by the following property: $\ell\in S$ if and only if $\ell\mid\#G_p$ and $\ell^{-1}\notin R$. Then $\sum_{a\in G_p} \lambda^+(a,p) = 0$ and for all $\ell\in S$ we have that
		\begin{equation}\label{eqconjintro}
			\prod_{a\in G_p} \pi_{\ell}(a)^{\lambda^+(a,p)} \equiv \pi_{\ell}(\tilde{q}_p)^{\frac{\lambda^+(0,1)}{2\tord_p(q_p)}},
		\end{equation}
	where $\pi_{\ell}$ denote the natural projection $\pi_{\ell}:G_p \rightarrow \tSyl_{\ell}(G_p)$ given by the Sylow decomposition of $G_p$. Furthermore, if $\trk_{\bZ}(E(\bQ)) > 0$, then for all $\ell\in S$, both sides of equation \eqref{eqconjintro} are $1$.
\end{conjecture}
\begin{remark}
	Note that both sides of equation \eqref{eqconjintro} are well-defined. This comes from the fact that if $\frac{a}{b}\in R$, then $(b,\ell) = 1$ for all $\ell\in S$. So, the map $g\mapsto g^b$ is an automorphism in $\tSyl_{\ell}(G_p)$ for all $\ell\in S$. Therefore, the element $g^{\frac{1}{b}}$ can be defined as the preimage of $g$ under the map $g\mapsto g^b$ (see Lemma \ref{Lemma2:2} for a more detailed explanation). 
\end{remark}

Conjecture \ref{conjintro} can be thought as a ``multiplicative'' formualtion of Conjecture 6 of \cite{MT87}. Moreover, Lemma \ref{equivalence_conjectures} implies that Conjecture \ref{conjintro} is equivalent to Conjecture $6$ of \cite{MT87}, when $E$ has split multiplicative reduction at $p$ and the layer of the Mazur-Tate element is $p$.
\vsep
We used SageMath to verify numerically this conjecture for around $400,000$ pairs $(E,p)$, where $E$ is a rational elliptic curve with split multiplicative reduction at the prime $p$. Interestingly, our findings indicate that Conjecture \ref{conjintro} does not hold in general. In our observations, the failure of Conjecture \ref{conjintro} occurs when there exist primes $\ell \in S$ such that $\ell \mid \# E(\bQ)_{\tTor}$ (where $E(\bQ)_{\tTor}$ denotes the torsion subgroup of $E(\bQ)$). However, the converse is not true. There are cases when there exists primes $\ell\in S$ such that $\ell\mid \# E(\bQ)_{\tTor}$ and equation \eqref{eqconjintro} still holds for all $\ell\in S$. The failure of equation \eqref{eqconjintro} for some primes $\ell\in S$ implies that a larger subring $R\subseteq \bQ$ is needed for Conjecture \ref{conjintro} to hold (see Remark \ref{importantRemark} for an explanation of the effect that the ring $R$ has on Conjecture \ref{conjintro}).
\vsep
Also, the fact that Conjecture \ref{conjintro} does not hold in general, implies that Conjecture $6$ of \cite{MT87} also does not hold in general, because they are equivalent if $E$ has split multiplicative reduction at $p$ (see Lemma \ref{equivalence_conjectures}). Furthermore, the degree of failure does entail that Conjecture $4$ of \cite{MT87} also does not hold in general (see the discussion following Conjecture \ref{conj4mul} and Section $3$). In Section $4$, we present two examples of the calculations carried out to check Conjecture \ref{conjintro} and Conjecture \ref{conj4mul}.
\vsep
In spite of this, by adding the hypothesis that $E(\bQ)_{\tTor}^{-1}\in R$ in Conjecture \ref{conjintro} we get a conjecture that does appear to hold, see Conjecture \ref{conjintromod}. We present all our numerical findings in Section 3. 
\vsep
Finally, all the code and documentation used in this article can be found on: \url{https://github.com/JpLlerena/Numerical_study_of_refined_conjectures_of_the_BSD-type}.
\begin{remark}\label{introremark}
Even though numerical evidence of the Conjectures stated in \cite{MT87} does appear to be largely absent in the literature, there exists theoretical results regarding the conjectures stated in \cite{MT87}. More specifically, the condition that $\sum_{a\in G_p} \lambda^+(a,p) = 0$, in Conjecture \ref{conjintro}, was proven by Bergunde and Gehrmann (see the Theorem in the introduction of \cite{BL17}). Also, de Shalit \cite{dSh95} proved that if $E$ has prime conductor, then there exists a subset $S' \subseteq S$ such that equation \eqref{eqconjintro} holds for all $\ell \in S'$ (see Remark \ref{rmkdShalit} for the precise definition of $S'$). Our findings do not contradict any of these results.
\end{remark}
\vsep
The structure of the article is as follows:
\vsep
The first section consists of preliminary results regarding the augmentation ideal of a group algebra. The main objective of this section is to prove Corollary \ref{cor2:1}, which allows an easier implementation of some of the refined conjectures of the Birch-Swinnerton-Dyer type in SageMath.
\vsep
In Section 2, we recall some of the refined conjectures of the Birch-Swinnerton-Dyer type. Furthermore, we prove Lemma \ref{equivalence_conjectures}, which relates Conjecture \ref{conjintro} and Conjecture $6$ of \cite{MT87}. We also state ``multiplicative'' formulations of Conjecture 4 and Conjecture 5 of \cite{MT87} (see Conjecture \ref{conj4mul} and Conjecture \ref{conj5mul}).
\vsep
In Section 3, we present our numerical findings and a slight variation of Conjecture \ref{conjintro}, which does appear to hold (see Conjecture \ref{conjintromod}). Also, we state a slight variation of Conjecture $4$ and Conjecture $6$ of \cite{MT87} which do appear to hold (see Conjecture \ref{conj4mod} and Conjecture \ref{conj6mod}).
\vsep
In Section $4$, we explain, with two examples, the calculations carried out to check Conjecture \ref{conjintro} and Conjecture \ref{conj4mul}.
\vsep
\textit{Acknowledgments}: This article is partially based on the author's master's thesis, done under the supervision of Daniel Barrera, whom we would like to thank for introducing us to these beautiful conjectures. We would like to thank Mihir Deo, Antonio Lei, and Ricardo Toso for feedback on earlier versions of this preprint. The author is most grateful to Ricardo Toso for the fruitful discussions and encouragement during the writing of this article.
\vsep
\textit{During the work done in this article, the author was funded by the National Agency of Research and Development (ANID) - Human Capital Subdirectorate's National Masters scholarship 2022 N° 22221372.}

\section*{Notation and conventions}
We will denote by $\cP\subseteq \bN$ the set of prime numbers. If $G$ is a finite group and $\ell\in \cP$, then $\tSyl_{\ell}(G)$ will denote the $\ell$-Sylow subgroup of $G$, with the convention that if $\ell \nmid \#G$ then $\tSyl_{\ell}(G) = \{e\}$. All rings are assumed to be commutative and have a multiplicative identity. Also, if $R$ denotes a ring and $G$ a group, then $R[G]$ will denote the group algebra of $G$ with coefficients in $R$, and if $g\in G$, then $[g]$ will denote the element $g$ viewed in $R[G]$ \textit{i.e.} $[g]\in R[G]$. Finally, we will use the abbreviation BSD for ``Birch-Swinnerton-Dyer''.

\section{The augmentation ideal of a group algebra}
In this section, $G$ will always denote a finite abelian group (we will use the multiplicative notation for the operation of $G$) and $R$ will always denote a subring of $\bQ$. Also, given a rational number $\frac{a}{b}\in \bQ$ we will always assume that $(a,b) = 1$.
\vsep
The \textbf{augmentation map} $\mathrm{aug}:R[G] \rightarrow R$ is defined as the $R$-linear map uniquely determined by $[g] \mapsto 1; \hsep \forall g\in G$. The \textbf{augmentation ideal of $R[G]$} is defined as $I(R,G)\defeq \tker(\mathrm{aug})$. Also, let $Q_n(R,G) \defeq I(R,G)^n/I(R,G)^{n+1}$ for all $n\in \bN$. 
\begin{lemma}\label{prop1:1}
	Fix $n\in \bN$. Define
	\begin{align*}
		f: R \otimes_{\bZ} Q_n(\bZ,G) &\rightarrow Q_n(R,G)\\
		r\otimes\left(\sum_{g\in G} \alpha_g [g] + I(\bZ, G)^{n+1}\right) &\mapsto \sum_{g\in G} r\alpha_g [g] + I(R, G)^{n+1},
	\end{align*}
	If we consider $R \otimes_{\bZ} Q_n(\bZ,G)$ as an $R$-module, by extension of scalars, then the function $f$ is a well-defined isomorphism of $R$-modules.
\end{lemma}
\begin{proof}
	The fact that $f$ is a well-defined homomorphism is a direct calculation. The fact that $f$ is bijective is a consequence of the following isomorphism of $R$-modules
	\begin{align*}
		R\otimes_{\bZ} I(\bZ, G)^{n+1} &\rightarrow I(R, G)^{n+1}\\
		r \otimes \sum_{g\in G} \alpha_g [g] &\mapsto \sum_{g\in G} r\alpha_g [g]
	\end{align*}
	where $R \otimes_{\bZ} I(\bZ, G)^{n+1}$ is an $R$-module by extension of scalars.
\end{proof}
By the previous lemma, we can see that to study the group $Q_n(R, G)$ it is sufficient to study the effect of tensoring $Q_n(\bZ, G)$ by $R$. Additionally, we will show that the tensor product $R \otimes_{\bZ} Q_n(\bZ,G)$ can be described explicitly in terms of $\ell$-Sylow subgroups of $Q_n(\bZ,G)$ (see Proposition \ref{Prop2:2}). 
\vsep
Now, let us recall some well-known results.
\begin{lemma}\label{lemma2:1}
	For every $n\in \bN$ the quotient $Q_n(\bZ, G)$ is a finite abelian group.
\end{lemma}
\begin{proof}
	We have that $X = \{[g] - [e]; \hsep g\in G\}$	is a basis of $I(\bZ,G)$ (see \cite[Exercise 8, pag. 159]{Joh97}). So, $I(\bZ,G)^n$ is generated by finite products of elements of $X$. Therefore, $I(\bZ,G)^n$ is finitely generated. As a consequence, $Q_n(\bZ, G)$ is also finitely generated.
	\vsep
	Finally, because $Q_n(\bZ, G)$ is finitely generated and every element has finite order (see \cite[Proposition 2.1]{CG11}), we can conclude that $Q_n(\bZ, G)$ is a finite abelian group.
\end{proof}
\begin{lemma}\label{Lemma2:2}
	We have the following results:
	\begin{enumerate}
		\item\label{trivial1} If $G$ is a $p$-group and $p^{-1}\in R$ then $R \otimes_{\bZ} G$ is trivial.
		\item\label{nontrivial} If for all primes $p\mid\#G$ we have that $ p^{-1}\notin R$, then $G$ has a natural structure of $R$-module. Furthermore, if we endow $R\otimes_{\bZ} G$ with the structure of $R$-modules by extension of scalars, then $R\otimes_{\bZ} G \cong G$ as $R$-modules.
	\end{enumerate}
\end{lemma}
\begin{proof}
	Item (\ref{trivial1}) follows from a direct calculation. 
	\vsep
	For item $(\ref{nontrivial})$ we first have to define the structure of $R$-module on $G$. Note that for any $\frac{a}{b}\in R\subseteq \bQ$, we have that $(\# G, b) = 1$. Thus $g\mapsto g^b$ is an automorphism of $G$. So, for each $g\in G$ there exists a unique $h\in G$ such that $h^b = g$, such $h$ will be denoted by $g^{\frac{1}{b}}$. This allows us to define the action of $R$ in $G$ as follows: if $ r = \frac{a}{b}\in R$ and $g\in G$ then $r\cdot g \defeq g^{\frac{a}{b}} \defeq (g^{\frac{1}{b}})^a$. A direct calculation shows that this action gives $G$ a structure of $R$-module\footnote{An example of this $R$-module structure is the following: Let $G = \bZ/N\bZ$ for some $N\in\bN$. The structure of $R$-module is given by the usual multiplication i.e. if $\frac{a}{b}\in R$ then $\frac{a}{b}\cdot n \defeq \frac{an}{b} \pmod{N}$ for all $n\in \bZ/N\bZ$, and the fraction is well-defined because $(b,N) = 1$.}.
	\vsep
	Lastly, a direct calculation shows that, the map uniquely determined by $r\otimes g \mapsto g^r$ defines an isomorphism of $R$-modules between $R\otimes_{\bZ} G$ and $G$.
\end{proof}
\begin{proposition}\label{Prop2:2}
	Fix $n\in\bN$. Let $S\subseteq \cP$ be the set of prime numbers defined as follows: $\ell\in S$ if and only if $\ell \mid \#G$ and $\ell^{-1}\notin R$. Then
	\[
		Q_n(R,G) \cong \bigoplus_{\ell \in S} \tSyl_{\ell}\left(Q_n(\bZ,G)\right),
	\]
	as $R$-modules, where the structure of $R$-module of the right hand-side is given by item (\ref{nontrivial}) of the previous lemma.
\end{proposition}
\begin{proof}
	On one hand, by Lemma \ref{lemma2:1} we known that $Q_n(\bZ, G)$ is a finite abelian group, henceforth it has a natural decomposition in terms of its $\ell$-Sylow subgroups. So,
	\begin{equation}\label{isoeq}
		Q_n(\bZ,G) \cong \bigoplus_{\ell \in \cP} \tSyl_{\ell}\left(Q_n(\bZ,G)\right).
	\end{equation}
	On the other hand, we know that the tensor product distributes with the direct sum. Thus, tensoring by $R$ the previous equation, we get
	\[
		R \otimes_{\bZ} Q_n(\bZ,G) \cong R \otimes_{\bZ} \left(\bigoplus_{\ell \in \cP} \tSyl_{\ell}\left(Q_n(\bZ,G)\right)\right) \cong \bigoplus_{\ell \in \cP} R \otimes_{\bZ} \tSyl_{\ell}\left(Q_n(\bZ,G)\right).
	\]
	So, using Lemma \ref{prop1:1} and Lemma \ref{Lemma2:2} in the previous equation we can conclude the proposition.
\end{proof}
By Proposition \ref{Prop2:2}, we can see that, to understand $Q_n(R, G)$ it is sufficient to understand the non-trivial $\ell$-Sylow subgroups of $Q_n(\bZ, G)$ for the primes $\ell$ that are not invertible in $R$.
\vsep
For our purposes, we are mainly interested in the particular case of $Q_1(R, G)$. In this case, we can further describe $Q_1(R, G)$ in terms of the $\ell$-Sylow subgroups of $G$.
\begin{corollary}\label{cor2:1}
	Let $S\subseteq \cP$ be the set of prime numbers defined as follows: $\ell\in S$ if and only if $\ell \mid \#G$ and $\ell^{-1}\notin R$. Then
	\begin{equation}\label{isoeqR}
		Q_1(R,G) \cong \bigoplus_{\ell \in S} \tSyl_{\ell}\left(G\right),
	\end{equation}
	as $R$-modules, where the structure of $R$-module of the right handside is given by item (\ref{nontrivial}) of Lemma \ref{Lemma2:2}.
\end{corollary}
\begin{proof}
	We know that the map $\varphi: Q_1(\bZ, G) \rightarrow G$ defined by
	\begin{equation}\label{defiso}
		\varphi\left(\sum_{g\in G} \beta_g [g] + I(\bZ,G)^2 \right) = \prod_{g\in G} g^{\beta_g}
	\end{equation}
	is an isomorphism (see \cite[Prop 2, p. 150]{Joh97}\footnote{We remark that Johnson's book uses $U$ for denoting $I(\bZ, G)$.}). Therefore, using the fact that $\varphi$ is an isomorphism in  Proposition \ref{Prop2:2}, with $n = 1$, we can conclude the proof.
\end{proof}
\begin{remark}\label{iso}
	Note that we can describe explicitly an isomorphism for equation \eqref{isoeqR}. We will use the  same notation as in Corollary \ref{cor2:1}. Also, denote by $\pi_{\ell}:G \rightarrow \tSyl_{\ell}\left(G\right)$ the natural projection given by the Sylow decomposition. An isomorphism between $Q_1(R,G)$ and $\bigoplus_{\ell \in S} \tSyl_{\ell}\left(G\right)$ is 
	\begin{align*}
		\phi: Q_1(R,G) &\rightarrow \bigoplus_{\ell \in S} \tSyl_{\ell}\left(G\right)\\
		\sum_{g\in G} r_g [g] + I(R,G)^2 &\mapsto \left(\prod_{g\in G} \pi_{\ell}(g)^{r_g}\right)_{\ell\in S} \ ,
	\end{align*}
	where $(-)_{\ell\in S}$ denotes a tuple in $\bigoplus_{\ell \in S} \tSyl_{\ell}\left(G\right)$ indexed by $S$. This isomorphism is obtained by composing the isomorphisms from Lemma \ref{prop1:1}, equation \eqref{isoeq}, item (\ref{nontrivial}) of Lemma \ref{Lemma2:2}, and $\varphi$ (which is defined in equation \eqref{defiso}).
\end{remark}

\section{Refined conjectures of the BSD type}
Let $E/\bQ$ be an elliptic curve of conductor $N\in \bN$. By the work of Tate \cite{Tat74}, we know that if $E$ has split multiplicative reduction at a prime $p$, then there exists a unique $q_p\in \bQ_p^*$ such that
\[
	E(\bC_p) \cong \bC_p^*/q_p^{\bZ},
\]
as rigid analytic spaces and $\tord_p(q_p) > 0$, where $\tord_p$ denotes the $p$-adic valuation, normalized at $p$ \textit{i.e.} $\tord_p(p) = 1$. We will refer to the $q_p$ of the previous equation as the \textbf{$p$-adic period of $E$}.
\vsep
We will denote by $\omega_E$ a N\'eron differential of $E$ and by $\Lambda_E$ the N\'eron lattice of $E$ i.e. $\Lambda_E \defeq \{\int_{[\gamma]} \omega_E; \hsep [\gamma]\in H_1(E(\bC), \bZ)\}$. Following the convention in \cite{MT87}, we define $\Omega_E \defeq \frac{1}{2} \int_{E(\bR)} |\omega|$, which may differ from the smallest positive real period of $\Lambda_E$ by a factor of $\frac{1}{2}$ depending if $\Lambda_E$ is rectangular or non-rectangular.
\vsep
Denote by $\Gamma_0(N)$ the Hecke subgroup of level $N$. Also, let $\cH$ denote the upper-half plane and $\cH^* \defeq \cH \cup \bP^1(\bQ)$ the extended upper-half plane.
\vsep
In 2001, the Modularity Theorem was proven by the work of Breuil, Conrad, Diamond, and Taylor building upon the work of Wiles (see \cite{Wil95}, \cite{TW95}, \cite{BCDT01}). As a consequence of the Modularity Theorem, we know that there exists a complex uniformization $\varphi: \Gamma_0(N)\setminus\cH^* \rightarrow E(\bC)$ such that 
\[
	\varphi^*(\omega_E) = 2\pi i c_E f_E(z) dz,
\]
where $f_E\in S_2(\Gamma_0(N))$ is a newform of weight $2$ and $c_E\in \bQ$ is the Manin constant, which is well-defined up to sign.
\begin{definition}[{\textit{c.f.} \cite[Definition 1.1]{MR23}}]\label{modsymbdef}
	We define the ``$+$'' modular symbol attached to $E$ at $\frac{a}{b}\in \bQ$ as
	\[
	\lambda^+(a,b) \defeq \frac{\pi i}{\Omega_E}\left(\int_{i\infty}^{\frac{a}{b}} f_E(z) dz + \int_{i\infty}^{\frac{-a}{b}} f_E(z) dz\right).
	\]
	For simplicity, we will omit the superscript $+$ \textit{i.e.} $\modsymb{a}{b} \defeq \lambda^+(a,b)$.
\end{definition}
Note that $\lambda(a,b) = \lambda(-a,b)$ for all $\frac{a}{b}\in \bQ$. Also, by the Theorem of Manin-Drinfel'd (\cite{Man72}, \cite{Dri73}) we know that $\modsymb{a}{b}\in \bQ$ for all $\frac{a}{b}\in \bQ$.
\begin{lemma}\label{translation_invariant}
	For any $\frac{a}{b}\in \bQ$ we have that $\lambda(a,b) = \lambda(a + b,b)$.
\end{lemma}
\begin{proof}
	This follows from a direct calculation, using the fact that $f_E(z)$ is invariant under the map $z \mapsto z + 1$. 
\end{proof}
\vsep
Given $M\in \bN$ let $G_M \defeq \GU{M}/\left<-1\right>$.
\begin{definition}[{\textit{c.f.} \cite[Section 1.2]{MT87}}]\label{Mazur-Tate}
Fix an elliptic curve $E/\bQ$ and a natural number $M\in \bN$. Let $R\subseteq \bQ$ be a subring containing $\left\{\modsymb{a}{M}\right\}_{a\in G_M}$. We define the \textbf{Mazur-Tate element at layer $M$} as
\[
	\theta_{E, M} \defeq \sum_{a\in G_M} \modsymb{a}{M} [a]\in R\left[G_M\right].
\]
If the elliptic curve is implicitly clear, we will omit the subscript $E$. \textit{i.e.} $\theta_{M} \defeq \theta_{E, M}$.
\end{definition}\label{MTEl}
Note that $\modsymb{a}{M}$ is independent on the choice of coset in $G_M = \GU{M}/\left<-1\right>$. This follows from Lemma \ref{translation_invariant} and using the fact that $\lambda(a,M) = \lambda(-a,M)$. Therefore, the Mazur-Tate element is well-defined.
\vsep
We define the vanishing order of $\theta_M\in R[G_M]$ as
\[
	\tord(\theta_M) \defeq \begin{cases}
		0 & \text{ if } \theta_M \notin I(R, G_M),\\
		r & \text{ if }\theta_M\in I(R, G_M)^r-I(R, G_M)^{r+1},\\
		\infty & \text{ if } \theta_M\in I(R, G_M)^r; \forall r\in \bN.
	\end{cases}
\]
Fix $S_m\subseteq \cP$ a subset of primes numbers such that if $p\in S_m$ then $E$ has split multiplicative reduction at $p$. Also, for each $p\in S_m$ fix a positive number $e_p > 0$ and set
\[
	M \defeq \prod_{p\in S_m} p^{e_p}.
\]
For each $p\in S_m$ let $\tilde{q}_p \defeq q_p/p^{\tord_p(q_p)}\in \bZ_p^*$. Note that we have natural mappings
\[
	\bZ_p^* \twoheadrightarrow \left(\bZ/p^{e_p}\bZ\right)^* \hookrightarrow \left(\bZ/M\bZ\right)^* \twoheadrightarrow G_M .
\]
This allows us to see $\tilde{q}_p$ as an element of $G_M$.
\vsep
Now, set $r \defeq \# S_m$. We will numerically study the following conjectures.
\begin{conjecture}[{\textit{c.f.} \cite[Conjecture 4]{MT87}}]\label{conj4}
	Let $R\subseteq \bQ$ be the smallest subring containing $\left\{\modsymb{a}{M}\right\}_{a\in G_M}$. We have that $\tord(\theta_M)\geq r + \trk_{\bZ}(E(\bQ))$.
\end{conjecture}
\begin{conjecture}[{\textit{c.f.} \cite[Conjecture 5]{MT87}}]\label{con5}
	 Assume that $\tau \defeq \#E(\bQ)$ is finite. Let $R\subseteq \bQ$ be the smallest subring containing $\tau^{-1}$ and $\left\{\modsymb{a}{M}\right\}_{a\in G_M}$. Then, we have that $\theta_M\in I(R,G_M)^r$ and if we denote by $\tilde{\theta}_M$ the image of $\theta_M$ in $Q_r(R,G_M)$ then
		\[
			\tilde{\theta}_M = \left(\prod_{p\in S_m} [\tilde{q}_p] - [1]\right) \dfrac{\#\Sha_E \prod_{p\in \cP-S_m} C_p}{\tau^2} \text{ in } Q_r(R, G_M),
		\]
		where $\Sha_E$ denotes the Tate-Shafarevich group, $C_p \defeq [E(\bQ_p):E^0(\bQ_p)]$, and $E^0$ denotes the identity component.
\end{conjecture}
\begin{conjecture}[{\textit{c.f.} \cite[Conjecture 6]{MT87}}]\label{con2}
	Let $R\subseteq \bQ$ be the smallest subring containing $\frac{\modsymb{0}{1}}{2\prod_{p\in S_m} C_p}$ and $\left\{\modsymb{a}{M}\right\}_{a\in G_M}$. Then we have that $\theta_M\in I(R,G_M)^r$ and if we denote by $\tilde{\theta}_M$ the image of $\theta_M$ in $Q_r(R,G_M)$ then
		\begin{equation}\label{conjecture3:1}
			\tilde{\theta}_M = \left(\prod_{p\in S_m} [\tilde{q}_p] - [1]\right) \frac{\modsymb{0}{1}}{2\prod_{p\in S_m} C_p} \text{ in } Q_r(R, G_M).
		\end{equation}
		If $\trk_{\bZ}(E(\bQ)) > 0$, then both sides of equation \eqref{conjecture3:1} are $0$.
\end{conjecture}
\begin{remark}
	The term $\prod_{p\in S_m} ([\tilde{q}_p] - [1])$, in Conjecture \ref{con5} and Conjecture \ref{con2}, is the so-called ``corrected discriminant'' in \cite{MT87} after projecting to $Q_r(R, G_M)$ (see \cite[eq (2.6.5) p. 737]{MT87} for the general definition, and \cite[Conjecture 5]{MT87} for the particular case of Conjecture \ref{con5} and Conjecture \ref{con2}). 
\end{remark}
\begin{remark}\label{rmkR}
It should be noted that in the conjectures stated in \cite{MT87} there are few restrictions for the ring $R$. However, we will show that the ring $R$ has a significant impact on the veracity of the refined conjectures of the BSD type, see Remark \ref{importantRemark}
\end{remark}
As mentioned in the introduction, in the particular case when $S_m = \{p\}$ and $M = p$, we can reformulate Conjecture \ref{con2} to get a ``multiplicative'' formulation, which is easier to compute numerically. 
\vsep
\begin{lemma}\label{equivalence_conjectures}
	Conjecture \ref{conjintro} is equivalent to Conjecture \ref{con2} when $M = p$ and $E$ has split multiplicative reduction at $p$.
\end{lemma}
\begin{proof}
Firstly, note that the condition that $\theta_p\in I(R,G_p)$ is, by definition, the equality $\sum_{a\in G_p} \lambda(a,p) = 0$.
\vsep
Now, when $E$ has split multiplicative reduction at $p$ and $M = p$ we have that Conjecture \ref{con2} reads
\begin{equation}\label{preveq}
	\sum_{a\in G_p} \lambda(a,p) [a] = \tilde{\theta}_p = \left([\tilde{q}_p] - [1]\right) \frac{\modsymb{0}{1}}{2C_p} \text{ in } Q_1(R, G_p).
\end{equation}
Now, denote by $S\subseteq \cP$ the set of prime numbers defined as follows: $\ell\in S$ if and only if $\ell \mid \#G_p$ and $\ell^{-1}\notin R$. Also, denote by $\pi_{\ell}:G_p \rightarrow \tSyl_{\ell}(G_p)$ the natural projection given by the Sylow decomposition. We can use the isomorphism $\phi$ defined in Remark \ref{iso}, on equation \eqref{preveq}, to get
\begin{equation}\label{eqintroeq}
	\left(\prod_{a\in G_p} \pi_{\ell}(a)^{\lambda(a,p)}\right)_{\ell\in S} = \left(\pi_{\ell}\left(\tilde{q}_p\right)^{\frac{\lambda(0,1)}{2C_p}}\right)_{\ell\in S}
\end{equation}
where $(-)_{\ell\in S}$ denotes a tuple in $\bigoplus_{\ell\in S} \tSyl_{\ell}(G_p)$ indexed by $S$.
\vsep
Finally, using the fact that $C_p = \tord_p(q_p)$ (see \cite[Chap. V, \S 5, Pag. 438]{Sil94}, \cite[C.15, Thm. 15.1]{Sil09}, and \cite[Corollary 15.2.1]{Sil09}) in the previous equation we conclude the proof (note that Conjecture \ref{conjintro} is equation \eqref{eqintroeq} as a coordinate-wise equality). 
\end{proof}
\begin{remark}\label{rmkdShalit}
	In \cite{dSh95}, de Shalit proved that if $E$ has prime conductor, then Conjecture \ref{conjintro} holds when we replace $S$ by a subset $S'\subseteq S$ (see \cite[Theorem 0.3]{dSh95}). More specifically, define $S'\subseteq \cP$ as follows: $\ell \in S'$ if and only if $\ell > 3, \ell\mid \#G_p$, $\ell^{-1} \notin R$, and $\ell$ is coprime to the modular degree of $E$. If $E$ has conductor $p$, then Conjecture \ref{conjintro} holds after replacing $S$ by $S'$.
\end{remark}
We can carry out the same process with Conjecture \ref{con5}, and we deduce the following statement equivalent to Conjecture \ref{con5} when $S_m = \{p\}$ and $M = p$.
\begin{conjecture}\label{conj5mul}
	Assume $\tau \defeq \#E(\bQ)$ is finite and that $S_m = \{p\}$. Let $R\subseteq \bQ$ be the smallest subring containing $\tau^{-1}$ and $\{\modsymb{a}{p}\}_{a\in G_p}$. Define $S\subseteq \cP$ as follows: $\ell\in S$ if and only if $\ell \mid \#G_p$ and $\ell^{-1}\notin R$. Then $\theta_p\in I(R,G_p)$, and for all $\ell\in S$ we have that
	\[
		\prod_{a\in G_p} \pi_{\ell}(a)^{\modsymb{a}{p}} \equiv \pi_{\ell}(\tilde{q}_p)^{\dfrac{\#\Sha_E \prod_{p'\in \cP-\{p\}} C_{p'}}{\tau^2}},
	\]
	where $\pi_{\ell}:G_p \rightarrow \tSyl_{\ell}(G_p)$ denotes the natural projection given by the Sylow decomposition.
\end{conjecture}
Again, using the same procedure we have that, if $S_m = \{p\}$ and $\trk_{\bZ}(E(\bQ)) > 0$, then Conjecture \ref{conj4} implies the following conjecture.
\begin{conjecture}\label{conj4mul}
	Assume that $\trk_{\bZ}(E(\bQ)) > 0$ and that $S_m = \{p\}$. Let $R\subseteq \bQ$ be the smallest subring containing $\{\modsymb{a}{p}\}_{a\in G_p}$. Define $S\subseteq \cP$ as follows: $\ell\in S$ if and only if $\ell \mid \# G_p$ and $\ell^{-1}\notin R$. Then $\theta_p \in I(R,G_p)$, and for all $\ell\in S$ we have that
	\[
		\prod_{a\in G_p} \pi_{\ell}(a)^{\modsymb{a}{p}} \equiv 1,
	\]
	where $\pi_{\ell}:G_p \rightarrow \tSyl_{\ell}(G_p)$ denotes the natural projection given by the Sylow decomposition.
\end{conjecture}
Note that Conjecture \ref{conj4mul} is not equivalent to Conjecture \ref{conj4}, because the previous conjecture only allow us to conclude that the vanishing order of the Mazur-Tate element is at least $2$. So, in the case when $\trk_{\bZ}(E(\bQ)) > 1$ Conjecture \ref{conj4mul} is a weaker version of Conjecture \ref{conj4}.

\begin{remark}\label{importantRemark}
	Note that with this reformulation we can see the effect that the ring $R$ has on these conjectures. Because, as more primes $\ell\mid \#G_p$ are invertible in $R$, the more ``information'' we lose, as less $\ell$-Sylow subgroups are considered in the Sylow decomposition of $G_p$. For example if $p > 2$ and $(\frac{p-1}{2})^{-1}\in R$ then by the work of Bergunde-Gehrmann \cite{BL17} (see the Theorem in the introduction), Conjecture \ref{conj4mul}, Conjecture \ref{conj5mul}, and Conjecture \ref{conjintro} hold trivially. 
\vsep
	As mentioned in Remark \ref{rmkR}, in \cite{MT87} there are few restrictions on the ring $R$, other than containing the necessary values such that the conjectures are well-defined. So, depending on how one chooses the ring $R$, some of the conjectures in \cite{MT87} will hold trivially. This is the reason that, at the moment of stating the conjectures from \cite{MT87}, we consider $R$ to be the smallest ring such that the equations are well-defined, as to avoid trivializing the conjectures.
\end{remark}

\section{Numerical results}
Numerical evidence supporting the refined conjectures of the BSD type, or variations of them, does appear to be largely absent in the literature. The original conjectures were numerically verified by Mazur and Tate \cite[\S 3.2]{MT87}, and a slight variation of a conjecture stated by Mazur and Tate \cite[Conjecture 4]{MT87} was numerically studied by Portillo-Bobadilla in his Ph.D. thesis \cite{Pob04} and in a subsequent article \cite{Pob19} (for the precise conjectures that Portillo-Bobadilla studied see Conjecture 3.2 and Conjecture 3.3 of \cite{Pob19}).
\vsep
To expand the numerical evidence of the refined conjectures of the BSD type we used SageMath to check Conjecture \ref{conj5mul}, Conjecture \ref{conj4mul}, and Conjecture \ref{conjintro}. For that, we calculated the necessary values needed to check these conjectures for $425,713$ pairs of the form $(E,p)$ where $E$ is a rational elliptic curve with split multiplicative reduction at the prime $p$. We will refer to the pairs calculated as ``database''\footnote{This was done on a personal computer using SageMath version 10.1.}. 
\vsep
The database was created by taking the list of elliptic curves given by Cremona in \cite{eclibdata}, and for each elliptic curve we considered all the primes for which the elliptic curve has split multiplicative reduction. This systematic check was done for all elliptic curves up to conductor 50,000, and a non-complete list was done for elliptic curves with conductor between 50,000 - 91,000. The elliptic curve with the highest conductor in the database has conductor 90,134\footnote{It should be noted that the reason to not continue the calculations for larger conductors, was mearly due to time restriction. This is because, as the conductor of an elliptic curve gets larger, calculating the modular symbols takes a significant amount of time.}.

Now, we will summarize our results\footnote{We use the \textit{L-functions and modular forms database} (LMFDB) convention for labeling elliptic curves. Note that this may disagree with the Cremona labeling convention from \cite{eclibdata}}:
\begin{itemize}
	\item Regarding Conjecture \ref{conjintro} we found $886$ pairs $(E,p)$ which did not satisfy the equality, some of these pairs are:
	\begin{itemize}
		\item The elliptic curve $130.a2$ and the prime $5$.
		\item The elliptic curve $680.c1$ and the prime $5$.
		\item The elliptic curve $798.d6$ and the prime $19$.		
		\item The elliptic curve $1890.i2$ and the prime $7$.
	\end{itemize}
	\item Regarding Conjecture \ref{conj4mul} we found $367$ pairs $(E,p)$ which did not satisfy the equality, some of these pairs are:
	\begin{itemize}
		\item The elliptic curve $377.a2$ and the prime $29$.
		\item The elliptic curve $832.f1$ and the prime $13$.
		\item The elliptic curve $4123.b1$ and the prime $7$.
		\item The elliptic curve $7826.b1$ and the prime $43$.
	\end{itemize}
	\item No counter-examples were found for conjecture \ref{conj5mul}.
\end{itemize}
For the full list see 

\url{https://github.com/JpLlerena/Numerical_study_of_refined_conjectures_of_the_BSD-type}

\begin{remark}
	We remark that when $2^{-1}\in R$, $\tord_p(q_p) = 1$, and $\lambda(0,1)$ is an integer, the so-called ``refined conjecture'' stated in \cite[Introduction, p. 712]{MT87} implies Conjecture \ref{conjintro}. To see this, compose the projections $\left(\bZ/p\bZ\right)^* \rightarrow G_p$ and $G_p \rightarrow \bigoplus_{\ell \in S} \tSyl_{\ell}(G_p)$ (where $S$ is defined as in Conjecture \ref{conjintro}). So, because some of the pairs $(E,p)$ that did not satisfy Conjecture \ref{conjintro} have the property that $2^{-1}\in R$, $\tord_p(q_p) = 1$, and $\lambda(0,1)$ is an integer (e.g. $1890.i2$ and the prime $7$), we can conclude that the refined conjecture also does not hold in general.
\end{remark}

\begin{remark}
	As mentioned in Remark \ref{rmkdShalit} and Remark \ref{introremark}, it is known that some cases of Conjecture \ref{conjintro} and Conjecture \ref{conj4mul} hold. So, it should be noted that our finding do not contradict the work of de Shalit \cite{dSh95}, nor the work of Bergunde-Gehrmann \cite{BL17}, as our finding fall outside the scope of their results.
\end{remark}
As mentioned in Remark \ref{rmkR}, the ring $R\subseteq \bQ$ does not have many restrictions in \cite{MT87}. However, by the numerical calculation that we found, if we consider $R$ to be the smallest ring such that Conjecture \ref{conj4mul} and Conjecture \ref{conjintro} are well-defined, then these two conjectures do not hold in general.
\vsep
Now, we reverify Conjecture \ref{conjintro} and Conjecture \ref{conj4mul} after adding the hypothesis that $(\#E(\bQ)_{\tTor})^{-1}\in R$ (where $E(\bQ)_{\tTor}$ denotes the torsion subgroup of $E(\bQ)$) which may be superflous in some cases, and may trivialize the conjectures in others (see Remark \ref{importantRemark} for an explanation of the effect of $R$ on the conjectures). However, with this additional hypothesis, we observed that all the conjectures did hold for all the pairs $(E,p)$ in the database, even for the pairs $(E,p)$ that previously did not satisfy Conjecture \ref{conjintro} or Conjecture \ref{conj4mul}. This can be thought as the failure of Conjecture \ref{conjintro} and Conjecture \ref{conj4mul} lies in the $\ell$-Sylow subgroups of $G_p$ for the primes $\ell \mid \#E(\bQ)_{\tTor}$ .
\begin{remark}
	We remark that the restriction that $(\#E(\bQ)_{\tTor})^{-1}\in R$ is mentioned in Conjecture $6$ in \cite{MT87}, but its not added as a hypothesis. However, our findings indicate that this hypothesis is necessary, not only for Conjecture $6$ in \cite{MT87} but also necessary for Conjecture $4$ in \cite{MT87}. 
\end{remark}
This naturally leads us to state the following variation of Conjecture \ref{conjintro}.
\begin{conjecture}\label{conjintromod}
	Assume that $E$ has split multiplicative reduction at $p$. Let $R\subseteq \bQ$ be the smallest subring containing $\frac{\modsymb{0}{1}}{2\tord_p(q_p)}$, $\{\modsymb{a}{p}\}_{a\in G_p}$ and $(\#E(\bQ)_{\tTor})^{-1}$. Define $S\subseteq \cP$ as follows: $\ell\in S$ if and only if $\ell \mid \# G_p$ and $\ell^{-1}\notin R$. Then $\sum_{a\in G_p} \lambda(a,p) = 0$ and for all $\ell\in S$ we have that
		\[
		\prod_{a\in G_p} \pi_{\ell}(a)^{\lambda(a,p)} \equiv \pi_{\ell}(\tilde{q}_p)^{\frac{\lambda(0,1)}{2\tord_p(q_p)}},
	\]
	where $\pi_{\ell}$ denote the natural projection $\pi_{\ell}:G_p \rightarrow \tSyl_{\ell}(G_p)$ given by the Sylow decomposition. Furthermore, if $\trk_{\bZ}(E(\bQ)) > 0$, then for all $\ell\in S$ both sides are $1$.
\end{conjecture}
We are naturally led to believe that the following variation of Conjecture \ref{con2} and Conjecture \ref{conj4} may hold.
\begin{conjecture}\label{conj6mod}
	Let $r \defeq \#S_m$ and denote by $R\subseteq \bQ$ the smallest subring that contains $\{\modsymb{a}{M}\}_{a\in G_M}$, $\frac{\modsymb{0}{1}}{2\prod_{p\in S_m} C_p}$, and $(\#E(\bQ)_{\tTor})^{-1}$. We have that  $\theta_M\in I(R,G_M)^r$ and if we denote by $\tilde{\theta}_M$ the image of $\theta_M$ in $Q_r(R,G_M)$ then
		\begin{equation}\label{eqcon6mod}
			\tilde{\theta}_M \equiv \left(\prod_{p\in S_m} [\tilde{q}_p] - [1]\right) \frac{\modsymb{0}{1}}{2\prod_{p\in S_m} C_p} \text{ in } Q_r(R, G_M),
		\end{equation}
		If $\trk_{\bZ}(E(\bQ)) > 0$, then both sides are $0$.
\end{conjecture}
\begin{conjecture}\label{conj4mod}
	Let $r \defeq \#S_m$. Assume that $R\subseteq \bQ$ is the smallest subring containing $\{\modsymb{a}{M}\}_{a\in G_M}$ and $(\#E(\bQ)_{\tTor})^{-1}$. Then $\tord(\theta_M)\geq r + \trk_{\bZ}(E(\bQ))$.
\end{conjecture}
\begin{remark}
	Note that we do not have to make any modification to Conjecture \ref{con5}, as the condition that $(\#E(\bQ)_{\tTor})^{-1}\in R$ is already imposed and we did not find any counter-examples.
\end{remark}
We want to emphasize that the evidence for Conjecture \ref{conj6mod} and Conjecture \ref{conj4mod} done in this article is only in a particular case \textit{i.e.} when the layer is $p$ and $E$ has split multiplicative reduction at $p$. Further calculations are needed to support these conjectures.
\begin{remark}
	We have to remark on the reliability of some arithmetic invariants in LMFDB and SageMath, in particular the order of the Tate-Shafarevich group\footnote{For more information about the reliability of the data, see \url{https://www.lmfdb.org/EllipticCurve/Q/Reliability}}. When verifying Conjecture \ref{conj5mul} we assumed that the order of the Tate-Shafarevich group given by SageMath was correct. However, in some cases, SageMath assumes the Strong BSD conjecture to calculate the order of the Tate-Shafarevich group. Thus, our results regarding Conjecture \ref{conj5mul} are correct ``up-to'' a possible flaw in the calculation of the order of the Tate-Shafarevich group.
\end{remark}
We also highlight that, at the moment of creating the database, we used the ECLIB implementation for calculating modular symbols, which is based on the work of Cremona \cite{Cre97}. However, for all the pairs $(E,p)$ in the database that did not satisfy one of the conjectures we recalculated the modular symbols, when possible\footnote{We had some technical difficulties with some of the implementations and some pairs $(E,p)$ in our database. For more information about these limitations, see the ``Disclaimer'' section on \url{https://github.com/JpLlerena/Numerical_study_of_refined_conjectures_of_the_BSD-type}}, using two other implementations in SageMath, one being based on the work of Stein \cite{Ste07}, and the another based on the work of Wuthrich \cite{Wut18}\footnote{This was done on the latest version of SageMath available, at the moment of writing, SageMath 10.4.}. We did not find any discrepancies in the values given by these three implementations.
\vsep
All the code used in this article, together with detailed explanations, can be found in the following GitHub repository 

\small{\url{https://github.com/JpLlerena/Numerical_study_of_refined_conjectures_of_the_BSD-type}}.
\section{Examples}
In this section, we give two examples of how the calculations can be done to check the conjectures, one for Conjecture \ref{conjintro} and one for Conjecture \ref{conj4mul}. We will use the same notation as in the previous sections.
\vsep
For Conjecture \ref{conj4mul}, consider the case of the elliptic curve $4123.b1$ and the prime $7$. Using the LMFDB website \footnote{\url{https://www.lmfdb.org/}}, we can see that this elliptic curve has split multiplicative reduction at $7$. Using the following command in SageMath, we obtain the values of the modular symbols $\lambda(1,7), \lambda(2,7)$, and $\lambda(3,7)$:
\begin{lstlisting}[language=Python,  style=mystyle]
	Input: 
	M = EllipticCurve("4123.b1").modular_symbol(+1, implementation='eclib')
	(M(1/7), M(2/7), M(3/7))
	Sagemath Output:
	(0, -1/2, 1/2)
\end{lstlisting}
However, this elliptic curve has negative discriminant, so its N\'eron lattice is non-rectangular. In this case, SageMath consider the smallest real positive period, which differs from the convention in \cite{MT87} of considering half the real period when the N\'eron lattice is non-rectangular. Therefore, we have to multiply the previous values by $2$. This allows us to see that all the modular symbols are integers. Thus $R = \bZ$.
\vsep
Calculating the expected equality of Conjecture \ref{conj4mul} we get the contradiction:
\begin{equation}\label{ex1eq}
	5 \equiv 1^{0}2^{-1}3^{1} \equiv 1 \hsep \hsep \text{ in } \hsep \left(\bZ/7\bZ\right)^*/\left<-1\right>. 
\end{equation}
However, we have that $5^3 \equiv 1^3$ in $\left(\bZ/7\bZ\right)^*/\left<-1\right>$. So, we get an equality if we cube both sides of equation \eqref{ex1eq}. This means that the failure of the conjecture lies in the $3$-Sylow subgroup (note that the kernel of the morphism $g \mapsto g^3$ in $G_7$ is precisly the $3$-Sylow subgroup). This coincides with the fact that $3$ divides $\#E(\bQ)_{\tTor} = 3$.
\vsep
In the case of Conjecture \ref{conjintro}, consider the case of the elliptic curve $680.c1$ and the prime 5. With the LMFDB website, we can see that this elliptic curve has split multiplicative reduction at $5$. Using the same command as before, to calculate $\lambda(0,1), \lambda(1,5)$, and $\lambda(2,5)$, we get:
\begin{lstlisting}[language=Python,  style=mystyle]
	Input: 
	M = EllipticCurve("680.c1").modular_symbol(+1, implementation='eclib')
	(M(0),M(1/5),M(2/5))
	Sagemath Output:
	(8, 0, 0)
\end{lstlisting}
In this case the discriminant is positive. Thus, we do not need to multiply these numbers by $2$.
\vsep
Now, with the following command in SageMath, we can get the $p$-adic expansion of the $p$-adic period (up to 5 digits):
\begin{lstlisting}[language=Python,  style=mystyle]
	Input: 
	EllipticCurve("680.c1").tate_curve(5).parameter(5)
	Sagemath Output:
	2*5^4 + 3*5^5 + 2*5^6 + 4*5^7 + 3*5^8 + O(5^9)
\end{lstlisting}
We can see that $\frac{\lambda(0,1)}{2\tord_5(q_5)} = 1$. Therefore, $R = \bZ$.
\vsep
Now, calculating the expected equality of Conjecture \ref{conjintro}, we get the following contradiction
\begin{equation}\label{lasteq}
	1^0 2^{0} \equiv 2^1\hsep \hsep \text{ in } \hsep \left(\bZ/5\bZ\right)^*/\left<-1\right>. 
\end{equation}
However, we can observe that $1^2 \equiv 2^2$ in $\left(\bZ/5\bZ\right)^*/\left<-1\right>$. Thus, after squaring equation \eqref{lasteq}, we get an equality. Again, this can be thought as the failure of the conjecture lies in the $2$-Sylow subgroup. This coincides with the fact that $2$ divides $\#E(\bQ)_{\tTor} = 2$. 
\bibliographystyle{alpha} 
\bibliography{Article_Thesis} 
\end{document}